\documentclass[12pt,reqno]{amsart}

\usepackage{nameref}
\usepackage{upgreek}     
\usepackage{enumitem}
\usepackage{graphicx}
\usepackage{mathrsfs}
\usepackage{amsthm}
\usepackage{amssymb}
\usepackage{amsmath,amsthm,amsfonts}
\usepackage{tikz-cd}
\usepackage{mathabx}
\usepackage{hhline}
\usepackage{textcomp}
\usepackage{amsmath,amscd}
\usepackage[all,cmtip]{xy}
\usepackage{mathtools}
\usepackage[margin=1in]{geometry}
\usepackage{lipsum}
\usepackage{extarrows}
\usepackage{upgreek}
\usepackage{bbm}
\usepackage[displaymath]{lineno}
\usepackage[singlespacing]{setspace}%
\usepackage[colorlinks=true,breaklinks=true,bookmarks=true,urlcolor=blue,
citecolor=violet,linkcolor=purple,bookmarksopen=false,draft=false]{hyperref}

\providecommand{\U}[1]{\protect\rule{.1in}{.1in}}
\usepackage{amsmath,amscd}

\makeatletter
\let\orgdescriptionlabel\descriptionlabel
\renewcommand*{\descriptionlabel}[1]{%
	\let\orglabel\label
	\let\label\@gobble
	\phantomsection
	\edef\@currentlabel{#1}%
	\let\label\orglabel
	\orgdescriptionlabel{#1}%
}
\makeatother

\theoremstyle{definition}
\newtheorem{theorem}{Theorem}[section]
\newtheorem{lemma}{Lemma}[section]
\newtheorem{remark}{Remark}[section]

\newtheorem{definition}{Definition}[section]
\newtheorem*{theorem*}{Theorem}
\newtheorem{corollary}{Corollary}[section]

\numberwithin{equation}{section}

\newcommand{\abs}[1]{\lvert#1\rvert}

\DeclareMathAlphabet{\mathpzc}{OT1}{pzc}{m}{it}

\newcommand{\ftt}[1] {\mathsf{#1}}

\newcommand{\pss}{\mathrm{({PS}})_c}
\newcommand{\dit}{\displaystyle{\int}}

\newcommand{\va}{\varphi}
\newcommand{\cs}{continuous }
\def\derat#1{{d \over dt} \hbox{\vrule width0.5pt 
		height 5mm depth 3mm${{}\atop{{}\atop{\scriptstyle t=#1}}}$}}

\newcommand{\dd}{{\tt D}}
\newcommand{\dt}[1]{{\tt d}{#1}}
\newcommand{\fs}[1]{\mathbbm {#1}}
\newcommand{\eu}[1]{\EuScript {#1}}
\newcommand\Set[2]{\{\,#1\mid#2\,\}}

\newcommand{\lc}{\mathsf{L}}

\newcommand{\mt}{\mathbbm {d}}

\newcommand{\set}[1]{\left\{#1\right\}}
\newcommand{\snorm}[2][]{\left\lVert#2\right\rVert_{#1}}

\newcommand{\Sem}[1]  {\textsf{Sem}(#1)}

\newcommand{\zero}[1]{\mathbf{0}_{#1}}

\newcommand{\rr}{\mathbb{R}}
\newcommand{\nn}{\mathbb{N}}

\newcommand{\vv}{\mathcal{V}}

\newcommand{\bl}[1] {\mathbf {#1}}

\DeclareMathOperator{\Lip}{Lip_{loc}}
\newcommand{\bb}{\mathcal{B}}
\DeclareMathAlphabet\EuScript{U}{eus}{m}{n}
\SetMathAlphabet\EuScript{bold}{U}{eus}{b}{n}
\newcommand\opn{\ensuremath{\mathrel{\mathpalette\opncls\circ}}}

\newcommand{\opncls}[2]{
	\ooalign{$#1\subseteq$\cr
		\hidewidth\raisefix{#1}\hbox{$#1{\stylefix{#1}#2}\mkern2mu$}\cr}}
\def\raisefix#1{
	\ifx#1\displaystyle
	\raise.39ex
	\else
	\ifx#1\textstyle
	\raise.39ex
	\else
	\ifx#1\scriptstyle
	\raise.275ex
	\else
	\raise.150ex
	\fi
	\fi
	\fi
}
\def\stylefix#1{
	\ifx#1\displaystyle
	\scriptstyle
	\else
	\ifx#1\textstyle
	\scriptstyle
	\else
	\ifx#1\scriptstyle
	\scriptscriptstyle
	\else
	\scriptscriptstyle
	\fi
	\fi
	\fi
}
\DeclareFontFamily{U}{mathx}{\hyphenchar\font45}
\DeclareFontShape{U}{mathx}{m}{n}{
	<5> <6> <7> <8> <9> <10>
	<10.95> <12> <14.4> <17.28> <20.74> <24.88>
	mathx10
}{}
\newcommand{\fr}{Fr\'{e}chet }

\newcommand{\bo}{\mathbf{B}^{\mathrm{op}}}

\newcommand{\pc}{\partial_{\mathrm{c}}}

\newcommand{\Cl}[1]{\overline{#1}}

\newcommand{\Id}{\mathop\mathrm{Id}\nolimits}

\renewcommand{\emptyset}{\varnothing}

\makeatletter
\@namedef{subjclassname@2020}{%
	\textup{2020} Mathematics Subject Classification}
\makeatother
\begin{document}

\title{Global Implicit Function Theorems and critical point theory in Fr\'{e}chet Spaces}

\author{Kaveh Eftekharinasab}
\thanks{This work was supported by grants from the Simons Foundation (1030291, 1290607, K.A.E)}
\address{Address: Topology lab.  Institute of Mathematics of National Academy of Sciences of Ukraine, 
	Tereshchenkivska st. 3,  01024, Kyiv, Ukraine}

\email{kaveh@imath.kiev.ua}
\thanks {}



\begin{abstract}
We prove two versions of a global implicit function theorem, which involve no loss of derivative, for Keller's $ C_c^1 $-mappings between arbitrary \fr spaces. Subsequently, within this framework, we apply these theorems to establish the global existence and uniqueness of solutions to initial value problems that involve the loss of one derivative. Moreover, we prove a Lagrange multiplier theorem by employing indirect applications of the global implicit function theorems through submersions and transversality. 
\end{abstract}

\maketitle
\section{Introduction}
Consider a nonlinear equation of the form:
\begin{equation}\label{eq:1}
\upphi (\ftt{e},\ftt{g})= \bl{0},
\end{equation}
where $ \mathtt{e}, \mathtt{g}$, and $\bl{0} $ belong to arbitrary \fr spaces, and $ \bl{0} $ represents the zero element. Our aim is to establish  sufficient conditions under which it is possible to  globally and uniquely solve Equation \eqref{eq:1} for $ \ftt{g} $ in terms of $ \ftt{e} $, with the solution mapping  $ \eu{K} $ being differentiable, such that  $ \upphi $ does not lose the derivative: $ \upphi $ and $ \eu{K} $ belong to the same class of differentiability. 

Such results are known as global implicit function theorems and have been extensively studied and applied in various frameworks, including purely topological settings, finite and infinite-dimensional Banach spaces. 
It would be impossible to provide a comprehensive account of the developments in these spaces, but we may refer to \cite{bolt, san, ish} for early contributions and to the survey \cite{gu} and its references for recent developments.

Interest in the broader context of \fr spaces has only recently started to gain traction. In \cite{k3,k4},  global inversion theorems, which are closely related to global implicit function theorems, have been obtained in these spaces. In \cite{k3}, the result is closely linked to the Nash-Moser implicit function theorem, necessitating mappings to be at least twice continuously differentiable,
and \fr spaces to be tame. However,  
in \cite{k3}, we consider  arbitrary \fr spaces, and mappings are  required to be only continuously differentiable, provided they are also local diffeomorphisms.

Motivated by these results and recent developments in critical point theory in \fr spaces (\cite{k1,k2,k3}),  this paper aims  to derive global implicit  function theorems,  which involve no loss of derivative, applicable to  arbitrary \fr space for mappings which are only continuously differentiable by employing methods of critical point theory.

In \cite{k}, it has been elegantly demonstrated how a mountain pass theorem can establish a global homeomorphism theorem in purely topological spaces. Subsequently, this method has been applied to derive global inversion and implicit function theorems in Hilbert spaces (\cite{id,id2}) and Banach spaces (\cite{ga,ga1}). Inspired  by the Banach case in \cite{ga}, we will further apply this approach to prove global implicit function theorems in \fr spaces.

First, let's outline the key features of the generalization. One significant distinction between \fr and Banach contexts, frequently encountered, lies in the topology: while a Banach space's topology is determined by a single norm, a \fr space's topology is defined by a family of seminorms. Consequently, growth estimates, which are crucial for formulating assumptions and proofs in many situations in the Banach setting, become unavailable in \fr spaces because the objects under consideration are not directly comparable.

Another concern is that, unlike Banach spaces, \fr spaces exhibit many non-equivalent notions of differentiability. 
Among them, we believe that Keller's $C_c^k$-differentiability  is the most suitable notion for our objectives.
This differentiability is equivalent to the well-developed Michael-Bastiani differentiability and facilitates the application of a crucial tool in critical point theory in infinite dimensional spaces--the Palais-Smale condition. 

By employing this notion of differentiability, we prove two slightly different versions of a global implicit function theorem: Theorem \ref{th:gift1} and Theorem \ref{th:gift2}. A key assumption in these theorems is that, for a mapping $ \va: \fs{E} \to \fs{F} $, at each point $ f \in \fs{F} $,  we associate a specific functional $  \eu{J}_f: \fs{E} \to \rr$ that satisfies a compactness condition of Palais-Smale type. In Theorem \ref{th:gift1}, this functional is assumed to be locally Lipschitz; however, in Theorem \ref{th:gift2}, it is Keller's $ C_c^1$. 
One reason for such considerations is that while the class of locally Lipschitz functions is broader compared to Keller's $C_c^1$-mappings, the non-smooth analysis of such functions is more technical than Keller's $C_c^1$-differentiability.

An implication of these theorems is a global inversion theorem, which is presented in Theorem \eqref{th:inv}.
The rest of the paper concerns the applications, starting with addressing the following initial value problem that involves the loss of one derivative:
\begin{equation}\label{eq:las}
y'(t)= \upphi ( t, y(t), e ),
\end{equation} 
where $\upphi$ is a Keller's $C_c^1$-mapping, and the initial conditions are fixed both in time and in arbitrary \fr spaces. 
In Theorem \ref{th:ivp}, using  Theorems \ref{th:gift1} and \ref{th:gift2}, we establish sufficient conditions for the global existence and uniqueness of the solution over the entire time interval. 

It is worth mentioning that the ODE \eqref{eq:las} is one of the significant and challenging problems in \fr spaces and has been considered in many works, see \cite{p}, \cite{vanthera}, and \cite{ivan}. However, the available results have been obtained under rather restrictive conditions. In contrast, our result is valid for arbitrary \fr spaces, where the mappings are only assumed to be continuously differentiable. The only assumption that may seem restrictive is the Palais-Smale condition, which could be challenging to verify in practice.

We then apply the global implicit function theorems to develop critical point theory in \fr spaces, with the objective of deriving a Lagrange multiplier theorem. In this regard, we will employ submersions and transversality. 

A Lagrange multiplier method involves finding critical points of a mapping subject to a set of constraints, which typically form a differentiable submanifold of an ambient space.

Applying submersions and transversality is a common approach to constructing submanifolds in the Banach case. However,
expanding beyond the Banach setting necessitates a robust concept of submersion to extend the typical assertions regarding submersion to manifolds modeled over locally convex spaces. For such manifolds, various non-equivalent mappings, including infinitesimally surjective, na\"{i}ve submersion, and submersion, are available for constructing submanifolds, see \cite{sh}.
In \cite{go}, submersions have been utilized to construct submanifolds in the case of manifolds modeled on locally convex spaces.
However, we will not use the strong notion of submersion as in \cite{go}. Instead, we will employ infinitesimally surjective mappings such that their tangent maps have split kernels. This concept coincides with the notion of submersion in the Banach setting. We also adapt the definition of transversality from the Banach setting.

For the mappings mentioned earlier, we propose a condition that involves the Palais-Smale condition (condition \eqref{s.1}), under which the mappings will, roughly speaking,  look like projections over their domains, as stated in Theorem \eqref{th:gst}. This condition plays a central role in formulating and proving a Lagrange multiplier theorem. Theorem \eqref{th:gst} has significant implications, including a preimage theorem (Corollary \ref{cor:1}), a transversality result (Corollary \ref{cor:2}), and a Lagrange multiplier theorem (Theorem \ref{th:lm}). 

Towards the end of the paper, we employ the Lagrange multiplier theorem to provide a brief generalization of the Nehari method for locating critical points of Keller's $ C_c^1 $-functionals. Additionally, we prove a theorem (Theorem \ref{th:ls}) that is particularly useful for identifying critical points of unbounded functionals. 

\section{Differentiability}
Throughout this paper, we assume that $ (\fs{F},\Sem{\fs{F}} )$ and $ (\fs{E},\Sem{\fs{E}} )$ are \fr spaces over $ \rr $, where $\Sem{\fs{F}}=\Set{\snorm[\fs{F},n]{\cdot}}{n \in \nn} $ and $\Sem{\fs{E}}=\Set{\snorm[\fs{E},n]{\cdot}}{n \in \nn} $ are families of continuous seminorms defining the topologies of $ \fs{F} $ and $ \fs{E} $, respectively, each directed by its natural order. We also use the translation-invariant metric
\begin{equation} \label{metric}
\mt_{\fs{F}}(x,y) = \sum_{n=1}^{\infty}\dfrac{1}{2^n}\dfrac{\snorm[\fs{F},n]{x-y}}{1 + \snorm[\fs{F},n]{x-y}}
\end{equation}
that induces the same topology on $ \fs{F} $.
We denote by $ \zero{\fs{F}} $ and $ \bo_{\mt_\fs{F}} $  the origin and
the open unit $ \mt_\fs{F} $-ball of $\fs{F}$, respectively. The Cartesian product $ \fs{E} \times \fs{F} $ is a \fr space with
seminorms defined by
$\snorm[\fs{E} \times \fs{F}, n]{(x,y)} \coloneqq
\max \set{\snorm[\fs{E}, n]{x} , \snorm[ \fs{F}, n]{y}}.$

As mentioned in the introduction, we will require the non-smooth analysis of locally Lipschitz mappings. In \cite{k1}, the critical points theory for these mappings, generalizing the Clarke subdifferential, has been developed. Now, we will revisit what will be needed later on.

Let $ \langle \fs{F}, \fs{F}' \rangle $ be a dual pairing. The weak topology $ \sigma(\fs{F},\fs{F}') $ on $ \fs{F} $ is defined through the family of seminorms:
\begin{equation}
\snorm[ A]{y} \coloneqq \sup_{x_i \in A} \abs{ \langle x_i,y \rangle}, \quad \text{for}\, y \in \fs{F},
\end{equation}
where $ A $ ranges over the set of finite subsets of $ \fs{F}' $. 

We further define the weak$^*$-topology $ \sigma^*(\fs{F}',\fs{F}) $ on $ \fs{F}' $ by the family of seminorms: 
\begin{equation}\label{eq:wss}
\snorm[ B]{x} \coloneqq \sup_{y_i \in B} \abs{ \langle y_i,x \rangle}, \quad   \text{for}\, x \in \fs{F}',
\end{equation}
where $ B  $ ranges over the set of finite subsets of $ \fs{F} $. 

We denote by $ \Lip (\fs{F},\rr)$  the set of locally Lipschitz functionals on $ \fs{F} $. We will refer to the following definitions, which can be found in \cite{k1}. For $ \varphi \in \Lip (\fs{F},\rr)$,  the generalized directional  derivative $ \varphi^{\circ}(x,y)  $ at each $ x \in \fs{F} $ in the direction $ y \in \fs{F}$ is defined by
\begin{equation}\label{def:gg}
\varphi^{\circ}(x,y)  \coloneqq \limsup_{h \rightarrow x, t\downarrow 0}  \frac{\varphi(h+ty)-\varphi(h)}{t}, \quad \text{for}\,t \in \rr, h \in \fs{F}.
\end{equation}
Additionally, we denote by $ \eu{F} : X \rightrightarrows Y $ a set-valued function from a set $ X $ to a set $ Y $. Specifically, for every $ x \in X $, $ \eu{F}(x) $ represents a subset (possibly empty) of $ Y $.

The Clarke subdifferential of $ \varphi $ at $ x $ is a set-valued function $ \pc \varphi: \fs{F} \rightrightarrows \fs{F}' $ defined by
\begin{equation*}
\pc \varphi (x) \coloneqq \Big\{ x' \in \fs{F}' \mid  \langle x',y \rangle \leq \va^{\circ}(x,y), \,\, \forall y \in \fs{F}\Big\}.
\end{equation*}
The set $ \pc \varphi (x) $ is weak$^*$-compact, ensuring the well-definedness of the following function:
\begin{gather}\label{eq:chpsc}
\lambda_{\varphi,B}: \fs{F} \to \rr, \quad
\lambda_{\varphi,B}(x) = 
\min_{y \in \pc \va(x)}  \snorm[B]{y}.	
\end{gather}
Here, $ B $ is a finite subset of $ \fs{F}' $.

We will need the following properties of locally Lipschitz mappings.
\begin{lemma}[\cite{k1}, Lemma 1.1]\label{lm:clsubdi}
	Let $ \va, \psi \in  \Lip (\fs{F},\rr)$, and let $ x \in \fs{F} $.
	\begin{enumerate}[label=$ \bl{(CS\arabic*)} $,ref=CS\arabic*]
		\item \label{itm:cds.1}  $ y \in \pc \varphi(x) $ if and only if $ \varphi^{\circ}(x,h)\geq \langle y,h\rangle, \, \forall h \in \fs{F} $. 
		\item \label{itm:cds.2} If $ x $ is a local extrema of $ \varphi $, then $ \zero{\fs{F}'} \in \pc \varphi(x) $.
	\end{enumerate}
\end{lemma}
Next, we define Keller's $C^k_c$-differentiability \cite{ke}. Throughout,
by $ U \opn \mathsf{T} $ we mean that $ U $
is an open subset of a topological space $ \mathsf{T} $.  If $ \mathsf{S}$ is another topological space, we denote by $ \eu{C} (\mathsf{T},\mathsf{S}) $ the set of  \cs mappings from  $ \mathsf{T} $ to $ \mathsf{S}$. Additionally,  $\lc(\fs{E},\fs{F})$ denotes the set of all continuous linear mappings from $\fs{E}$ to $\fs{F}$.

A bornology $\bb_{\fs{F}}$ on $\fs{F}$ is a  covering of $ \fs{F} $  that satisfies the following axioms:
\begin{enumerate}
	\item $\bb_{\fs{F}}$ is stable under finite unions.
	\item  If $ A \in \bb_{\fs{F}} $ and $ B \subseteq A$, then $ B \in \bb_{\fs{F}} $. 
\end{enumerate}
The compact bornology on $\fs{F}$, denoted by $\bb_{\fs{F}}^c$, consists of the family of relatively compact subsets of $\fs{F}$. This family is generated by the set of all compact subsets of $\fs{F}$, meaning that every $B \in \bb_{\fs{F}}^c$ is contained within some compact set.
Let $\bb_{\fs{E}}^c$ be the compact bornology on $ \fs{E} $.
We endow the vector space $\lc(\fs{E},\fs{F})$ with the $\bb_{\fs{E}}^c$-topology, which is the topology of uniform convergence on all compact subsets of $ \fs{E} $. This results in a Hausdorff locally convex topology defined by the family of  seminorms given by 
\begin{equation*}
\snorm[B,n]{L}   \coloneqq \sup_{e \in B } {\snorm[F,n]{L (e)}}, \quad \text{for}\, B \in \bb_{\fs{E}}^c,\, n \in \nn.
\end{equation*}	
Let $\varphi: U \opn \fs{E} \to \fs{F}$  be a mapping. If the directional  derivatives
\begin{equation*}
\varphi'(x)h = \lim_{ t \to 0} \dfrac{\varphi(x+th)-\varphi(x)}{t}
\end{equation*}
exist for all $x \in U$ and all $ h \in \fs{E} $, and  the induced map  $ \varphi'(x) : U \to \lc(\fs{E},\fs{F})$ is continuous for all
$x \in U$, then  we say that $ \varphi $ is of class Keller's $C_c^1 $, or simply a Keller's $C_c^1$-mapping. 
Higher-order derivatives are defined in usual manner.

For a continuous curve $ \gamma: I=(a,b) \to \fs{F}  $, we define its derivative as
\begin{equation*}
\gamma'(x) = \lim_{ t \to 0} \dfrac{\gamma(x+t)-\varphi(x)}{t}.
\end{equation*}
If the limit exists and is finite, and $ \gamma'(x) $ is continuous, we say that $ \gamma $ is $ C^1 $, a notion which coincides with  Keller's $C_c^1$-differentiability. If $ I=[a,b] $, the extension of the derivative by continuity of $ \gamma' $ to $ [a,b] $ has the values $ \gamma'(a) $ and $ \gamma'(b) $ equal to
\begin{equation*}
\gamma'(a) = \lim_{ t \downarrow 0} \dfrac{\gamma(a+t)-\varphi(a)}{t}, \quad \gamma'(b) = \lim_{ t \downarrow 0} \dfrac{\gamma(b)-\varphi(b-t)}{t}.
\end{equation*}

In \cite[Lemma 1.5]{k1}, it was shown that if $ \va $ is a Keller's $C_c^1 $-mapping, then $ \va'(x) \in \pc \va(x)$. We can easily prove that indeed $ \pc \va(x) = \set{\va'(x)} $.  Fix $ \bl{x} \in \fs{F} $. By the mean value theorem
\cite[Theorem 1]{kh}, for $ y $ close enough to $ x $ and $ t>0 $ close to $ 0 $, we obtain
$$
\dfrac{\va(y+t\bl{x}) - \va(y)}{t} = \langle \va'(z), \bl{x}\rangle
$$
for some $ z \in (y,y+t\bl{x}) $. As $ y \to x $ and $ t \downarrow 0 $, we have $ z \to x $.
Therefore, 
\begin{equation*}
\va^{\circ}(x,\bl{x}) \leq \langle \va'(x), \bl{x} \rangle
\end{equation*}
since $ \va' : \fs{F} \to 
\fs{F}' $ is continuous. If $ h \in \pc \va(x) $, then by Lemma \ref{lm:clsubdi}\eqref{itm:cds.1}, we have $$\va^{\circ} (h,\bl{x}) \leq \langle \va'(x),h\rangle .$$ As $ \bl{x} $ is arbitrary, we conclude  that $ \pc \va(x) = \set{\va'(x)} $.

\section{Global Implicit Function Theorems}
Our proof of a global implicit function theorem is constructive.  In essence,  to solve an equation 
\begin{equation}\label{eq:2}
\va (p,g)= \bl{0}, \quad 
\end{equation}
for a given $ g $, we associate a functional $ \eu{J}_g$ with $ \va $ in such a way that the solution of \eqref{eq:2} corresponds to a critical point of $ \eu{J}_g $. This is the most technically challenging aspect of our approach, involving the methods for locating critical points of functionals: minimization and the mountain pass theorems.

As mentioned in the introduction, we consider two classes of functionals: locally Lipschitz functionals and Keller's $ C_c^1 $-functionals. First, we focus on the locally Lipschitz case.

Our prime ingredient is a compactness condition of Palais-Smale type, for which  we need  the function $ \lambda $ defined in \eqref{eq:chpsc}.
\begin{definition}[\cite{k1}, Definition 2.1, Chang PS-Condition] Consider $ \varphi \in \Lip(\fs{F}, \rr) $.
	We say that $ \varphi $ satisfies
	the Palais-Smale condition in Chang's sense, or the Chang {PS}-condition, if every sequence $ (x_i) \subset \fs{F} $ such that  $ \varphi(x_i) $ is weakly$ ^{*} $ bounded and 
	\begin{equation}\label{eq:cpsc}
	\lim_{i \rightarrow \infty} \lambda_{\varphi,B}(x_i) = 0 \quad \text{for each finite subset}\, B \subset \fs{F}',
	\end{equation}
	possesses a convergent subsequence. Additionally, if  any sequence $ (x_i) \subset \fs{F} $ such that $ \varphi(x_i) \rightarrow c \in \rr$ and satisfies~\eqref{eq:cpsc} has a convergent subsequence, we say that $ \varphi $ satisfies the Chang {PS}-condition at level $ c $. 
\end{definition}
\begin{theorem}[\cite{k1}, Theorem 3.2, Mountain Pass Theorem]\label{th:mptll}
	Consider $ \varphi \in \Lip(\fs{F}, \rr) $ 
	and an open neighborhood $ \mathcal{U} $ of $ x \in \fs{F} $. Let $ y \notin \overline{\mathcal{U}} $ be  such that, for a real number $ m $,
	\begin{equation}\label{eq:ine}
	\max \{ \varphi(x), \varphi(y)\} < m \leq \inf_{\partial \mathcal{U}} \varphi.
	\end{equation}
	Suppose that $ \varphi $ satisfies \eqref{eq:ine} for a real number $ m $ and satisfies the Chang $\mathrm{ PS}$-condition at every level. Define 
	\begin{equation*}
	\Gamma \coloneqq \Big\{ \gamma \in \mathsf{C} ([0,1]; \fs{F}) \mid  \gamma (0)=  x,\gamma(1)=y \Big\},
	\end{equation*}	
	which is the  space of continuous paths joining $ x $ and $ y $. Let
	\begin{equation}\label{eq:mptllc}
	c \coloneqq \inf_{\gamma \in \Gamma} \max_{t \in [0,1]} \varphi (\gamma (t)) \geq m. 
	\end{equation} 
	Then, there exists a sequence $ (x_i) \subset \fs{F} $	such that $ \varphi(x_i) \rightarrow c $ and \eqref{eq:cpsc} holds.
	Moreover, as  $ \varphi $ satisfies the Chang $\mathrm{PS}$-condition at level $c$, we conclude that  $ c $ is a critical value of $ \varphi $.
\end{theorem}
\begin{lemma}[\cite{k1}, Lemma 4.1] \label{lem:mill}
	Consider $ \varphi \in \Lip (\fs{F}, \rr)$, which is bounded from bellow. Then, there exists a sequence $ (x_i) \subset \fs{F}$ such that 
	$ \lim_{i \rightarrow \infty} \varphi(x_i) =\inf_{\fs{F}} \varphi $, and 
	\begin{equation}
	\lim_{i \rightarrow \infty} \lambda_{\varphi,B}(x_i) = 0 \quad \text{for each finite subset}\, B \subset \fs{F}'.
	\end{equation}
\end{lemma}
For $ \varphi \in \Lip(\fs{F}, \rr) $, we define  $ x \in \fs{F} $ as a regular point of $ \va $ if the directional derivative $ \va'(x)h $ exists for all $ h \in \fs{F} $ and $ \va'(x)h = \va^{\circ}(x,h) $.
\begin{theorem}[\cite{k1}, Lemma 1.4, Chain Rule]\label{th:crll}
	Let  $ \psi : \fs{E}\rightarrow \fs{F}  $ be a Keller's $ C_c^1 $-mapping in an open neighborhood of $ e \in \fs{E} $, and $ \va : \fs{F} \rightarrow \rr $ a locally Lipschitz mapping.
	Then $ \Phi = \va \circ \psi  $ is locally Lipschitz, and 
	\begin{equation}\label{th:crll1}
	\pc \Phi (e) \subseteq  \pc \va (\psi (e)) \circ \psi' (e).
	\end{equation}
	Moreover, if $\va$ (or its negative $-\va$) is regular at $\psi(x)$, then $\Phi$ (or its negative $-\Phi$) is regular at $x$ and equality in \eqref{th:crll1} holds. Also, if $\psi$ maps every neighborhood of $x$ onto a set that is dense in the neighborhood $\va(x)$, then equality in \eqref{th:crll1} holds.	
\end{theorem}
The following theorem is inspired by the Banach case, \cite[Theorem 8]{ga1}. However, the proof and the assumptions are different. In the Banach case, norms are essential in both the hypothesis and the proof.
\begin{theorem}[Global Implicit Function Theorem I]\label{th:gift1}
	Let  $ \fs{G}$ be a Fr\'{e}chet space, and let $\varphi  : \fs{E} \times \fs{F} \rightarrow \fs{G} $ be a Keller's $ C_c^1 $-mapping. Assume $ \eu{I} : \fs{G}
	\rightarrow [0, \infty ]$ is a  locally Lipschitz function with the following two properties:
	\begin{enumerate}[label={\bf (GP\arabic*)},ref=GP\arabic*]
		\itemindent=6pt
		\item \label{itm:gdp1} $\eu{I}(x)=0$ if and only if $x=\zero{\fs{G}}$,
		\item \label{itm:gdp2} $\zero{\fs{G}'} \in \pc \eu{I}(y)$ if and only if $y =\zero{\fs{G}}$.
	\end{enumerate}
	Suppose that the following two conditions hold.
	\begin{enumerate}[label={\bf (GIF\arabic*)},ref=GIF\arabic*]
		\itemindent=10pt
		\item \label{itm:gift.1} For any $ f \in \fs{F} $, the function $\eu{J}_f: \fs{E} \to [0,\infty]$  defined by $ \eu{J}_f(e) = \eu{I} (\varphi (e,f) ) $ satisfies the Chang Palais-Smale condition  at all levels,
		\item \label{itm:gift.2} the partial derivative  in the first variable $ \dd_1 \varphi: \fs{E} \rightarrow \fs{G} $ is bijective. 
	\end{enumerate}
	Then, there exists a unique Keller's $ C_c^1 $-mapping $ \eu{K}: \fs{F}\rightarrow \fs{E} $ such that, for any $ g \in \fs{F} $, we have $ \varphi (\eu{K} (g),g)= \zero{\fs{G}}$. Moreover, the derivative $  \eu{K}'(g) $ is given by the formula
	\begin{equation}\label{eq:der1}
	\eu{K}' (g) = -\left[\dd_1 \varphi (\eu{K}(g),g) \right]^{-1} \circ \dd_2 \varphi (\eu{K}(g),g).
	\end{equation}
\end{theorem}
\begin{proof}
	Let $ g \in \fs{F} $ be given, and consider the function $ \eu{J}_g(e) = \eu{I} (\varphi (e,g) ) $.
	By Lemma \ref{lem:mill}, there exists a sequence $ (e_n)\subset \fs{E} $ such that 
	\begin{equation*}
	\lim_{n \rightarrow \infty} \eu{J}_g(e_n) =\inf_{\fs{E}} \eu{J}_g, 
	\end{equation*}
	and for each
	finite subset $ B $ of $ \fs{F}' $, we have
	\begin{equation}
	\lim_{n \rightarrow \infty} \lambda_{\eu{J}_g,B}(e_n) = 0.
	\end{equation}
	Since, by \eqref{itm:gift.1}, $ \eu{J}_g $ satisfies the Chang PS-condition, the sequence $ (e_n) $ has a convergent subsequent, once again denoted by 
	$(e_n)$,
	with the limit $ p $, which is a critical point of $ \eu{J}_g $. Thus,  by Lemma \ref{lm:clsubdi}\eqref{itm:cds.2}, we have $$ \zero{\fs{G}'} \in \pc \eu{J}_g (p). $$
	By the chain rule (Theorem \ref{th:crll}), we have 
	\begin{equation*}
	\pc \eu{J}_g (p) \subset \pc \eu{I} (\varphi (p,g))  \circ \dd_1 \varphi (p,g). 
	\end{equation*}
	Thus, there exists $ \xi \in \pc \eu{I} (\varphi (p,g)) $ such that $ \zero{\fs{G}'} = \xi \circ  \dd_1 \varphi (p,g)$. By \eqref{itm:gift.2}, the derivative $ \dd_1 \varphi $ is invertible at $ p $, implying that $ \xi = \zero{\fs{G}'} $. Therefore,  $ \zero{\fs{G}'} \in \pc \eu{I} (\varphi (p,g)) $, and hence \eqref{itm:gdp2} implies that 
	\begin{equation*}
	\varphi(p,g)= \zero{\fs{G}}. 
	\end{equation*}
	Now, we prove  by contradiction that $ p $ is the only point for which $\varphi(p,g)= \zero{\fs{G}}$. 
	Let $ e_1 \neq p \in \fs{E} $ be such that  $$ \varphi (e_1,g) =\varphi(p,g) = \zero{\fs{G}}. $$ 
	From the definition of the function  $ \eu{J}_g(e) = \eu{I} (\varphi(e, g)) $, it follows that $ \eu{J}_g (e_1) = \eu{J}_g(p) = 0 $.
	
	Let $ \mathbf{r} > 0 $ be  small enough  such that $ p \notin \Cl{e_1+\bl{r}\bo_{\mt_\fs{E}}} $.
	Without loss of generality, we can suppose $ e_1 = \zero{\fs{E}}$. For any $ e \in \partial (\zero{\fs{E}}+\bl{r}\bo_{\mt_\fs{E}}) $,  by  \eqref{itm:gdp1}, we have $\eu{I}(e) \neq 0 $, and therefore
	\begin{equation*}
	0 < m \leq \eu{J}_g (e) \quad \text{for some} \, m.
	\end{equation*}
	Thus, all assumptions of Theorem~\ref{th:mptll} hold; therefore, there exists $ ( e_n ) \subset\fs{ E} $ such that 
	\begin{equation*}
	\lim_{n \rightarrow \infty} \eu{J}_g (e_n) = c,
	\end{equation*}
	for some $ c \geq m  $ characterized by
	\eqref{eq:mptllc}. Since $ \eu{J}_g (e_n) $ satisfies the Chang PS-condition at $ c $, it has a convergent subsequent, denoted again by $ (e_n) $,  with the limit $h$. Therefore, $ h $ is a critical point of $\eu{J}_g$, and therefore $  \zero{\fs{G}} \in \pc  \eu{J}_g(h) $ by Lemma \ref{lm:clsubdi}\eqref{itm:cds.2}. Since 
	\begin{equation*}
	\lim_{n\rightarrow \infty} \eu{J}_g (e_n) = \eu{J}_g(h) = c  \geq m >  0,
	\end{equation*}
	it follows that
	\begin{equation}\label{eq:contr}
	\varphi(h,g) \neq \zero{\fs{G}}.
	\end{equation}
	By the chain rule (Theorem \ref{th:crll}), we have 
	\begin{equation*}
	\pc \eu{J}_g (h) \subset \pc \eu{I} (\varphi (h,g) ) \circ \dd_1 \varphi (h,g).
	\end{equation*}
	Therefore, there exists $v \in \pc \eu{I} (\varphi (h,g))$ such that 
	\begin{equation*}
	\zero{\fs{G}'} = v \circ  \dd_1 (h,g).
	\end{equation*}
	Since $ \dd_1 \varphi $ is invertible, it follows that $ v = \zero{\fs{G}'} $. Thus,  $\zero{\fs{G}'} \in \pc \eu{I} (\varphi (h,g))$ and hence \eqref{itm:gdp2}  implies that $\varphi(h,g)= \zero{\fs{G}} $  which contradicts \eqref{eq:contr}. 
	
	To conclude the proof, it is sufficient to define 
	\begin{equation*}
		\eu{K} (g) = p
	\end{equation*}
 for a given $ g \in \fs{F} $. Here, $ p $ is the solution to $ \varphi (p,g)= \zero{\fs{G}} $ obtained as above. The proof of Formula \eqref{eq:der1} is a straightforward  application of the chain rule.
\end{proof}
In the aforementioned theorem, it is assumed that the associated function $\eu{I} $ is locally Lipschitz, which is deemed advantageous.
However, non-smooth analysis is subtle and excessively technical, rendering it less practical. Thus, it would be also needed to assume that $\eu{I} $ is a Keller's $C_c^1$-mapping. Therefore, we will also consider this case.
While the approach remains similar, for the sake of clarity and comprehensiveness, we present it in full detail.

We will now revisit essential components related to Keller's $C_c^1$-mappings.
\begin{definition}[\cite{k2}, Definition 3.2, PS-Condition]\label{df:ps}
	Let $ \va : \fs{F} \to \rr $ be a Keller's $C_c^1$-mapping.
	We say that $\va$ satisfies the Palais-Smale condition, denoted as the $\mathrm{PS}$-condition, if every sequence $(x_i) \subset  \fs{F}$ for which
	$\va(x_i)$ is bounded and 
	\begin{equation*}
	\va'(x_i) \rightarrow 0 \quad \mathrm{ in} \quad  \fs{F}_{\mathrm{k}}' 
	\end{equation*}
	has a convergent subsequence.
	Additionally, we say that $\va$ satisfies the Palais-Smale condition at the level $c \in \rr$, the $(\mathrm{PS})_c$-condition, if each sequence $(x_i) \subset  \fs{F}$ for which
	\begin{equation*}
	\va(x_i) \to c \quad \mathrm{and} \quad \va'(x_i) \rightarrow 0 \quad \mathrm{ in} \quad  \fs{F}_{\mathrm{k}}',
	\end{equation*}		 		
	has a convergent subsequence.
\end{definition}
\begin{theorem}[\cite{k2}, Corollary 4.7]\label{co:minimizing}
	Let $ \va : \fs{F} \to \rr $ be a Keller's $C_c^1$-mapping that is bounded below.
	If the $(\mathrm{PS})_c$-condition holds with $c = \inf_{\fs{F}} \va$, then $\va$ achieves  its minimum at a critical point $ x_0 \in \fs{F} $ where $ \va(x_0)=c $.  
\end{theorem} 
\begin{theorem}[\cite{k3}, Theorem 2.3]\label{th:mpt} 
	Assume that $ \va: \fs{F} \to \rr $ is a Keller's $C_c^1$-mapping satisfying the $(\mathrm{PS})_c$-condition for ever $ c\in \rr $.  Let $ x_0 \in \fs{F} $. Consider an open neighborhood $U$ of $x_0 \in \fs{F}$, where $ \partial U$ denotes the boundary of $U$. 
	Assume that $x_1$ belongs to the distinct connected component  of $\fs{F} \setminus \partial{U} $. Suppose $\va$ satisfies the condition:
	\begin{equation}\label{mpt:geo}
	\inf_{p \in \partial U} \va(p) > \max \set{ \va(x_0),\va(x_1)} = a.
	\end{equation}
	Then $ \va $ has a critical value $ c > a $, which can be characterized as
	\begin{equation}\label{eq:mpt}
	c \coloneqq \inf_{\gamma \in \Gamma} \max_{t \in [0,1]} \va(\gamma(t)).
	\end{equation}
	Here, $$\Gamma \coloneqq \Big\{ \gamma \in \ftt{C} ([0,1], \fs{F})\mid \gamma (0)= x_0,\gamma(1)=x_1\in \fs{F} \Big\}$$ is the set of continuous paths joining $ x_0 $ and $ x_1 $.
\end{theorem}
\begin{theorem}[Global Implicit Function Theorem II]\label{th:gift2}
	Let  $ \fs{G}\; $ be  a Fr\'{e}chet space, and let $\varphi  : \fs{E} \times \fs{F} \rightarrow \fs{G} $ be a Keller's $ C_c^1 $-mapping. Suppose $ \eu{I} : \fs{G}
	\rightarrow [0, \infty ]$ is a Keller's $ C_c^1 $-mapping with the following two properties.
	\begin{enumerate}[label={\bf (SGP\arabic*)},ref=SGP\arabic*]
		\itemindent=14pt
		\item \label{itm:sgdp.1} $\eu{I}(x)=0$ if and only if $x=\zero{\fs{G}}$,
		\item \label{itm:sgdp.2} $\eu{I}'(y)=0$ if and only if $y =\zero{\fs{G}}$.
	\end{enumerate}
	Suppose that the following two conditions hold.
	\begin{enumerate}[label={\bf (SGIF\arabic*)},ref=SGIF\arabic*]
		\itemindent=18pt
		\item \label{itm:sgift.1} For any $ f \in \fs{F} $, the function $\eu{J}_f: \fs{E} \to [0,\infty]$  defined by $ \eu{J}_f(e) = \eu{I} (\varphi (e,f) ) $ satisfies the Palais-Smale condition at all levels,
		\item \label{itm:sgift.2} the partial derivative  in the first variable $ \dd_1 \varphi: \fs{E} \rightarrow \fs{G} $ is bijective. 
	\end{enumerate}
	Then,  there exists a unique Keller's $ C_c^1 $-mapping $ \eu{K}: \fs{F}\rightarrow \fs{E} $ such that, for any $ g \in \fs{F} $, we have $ \varphi (\eu{K} (g),g)= \zero{\fs{G}}$. Moreover, the derivative $  \eu{K}'(g) $ is given by the formula
	\begin{equation}\label{eq:der2}
	\eu{K}' (g) = -\left[\dd_1 \varphi (\eu{K}(g),g)\right]^{-1} \circ \dd_2 \varphi (\eu{K}(g),g).
	\end{equation}
\end{theorem}
\begin{proof}
	Let $ g \in \fs{F} $, and define the function $ \eu{J}_g(e) = \eu{I} (\varphi (e,g) ) $.
	By Theorem \ref{co:minimizing}, there exists a sequence $ (e_n)\subset \fs{E} $ such that 
	\begin{equation*}
	\lim_{n \rightarrow \infty} \eu{J}_g(e_n) =\inf_{\fs{E}} \eu{J}_g.
	\end{equation*}
	By \eqref{itm:sgift.1}, $ \eu{J}_g $ satisfies the PS-condition. Therefore, the sequence $ (e_n) $ has a convergent subsequent, denoted once again by 
	$(e_n)$,
	with the limit $ p $. This point is a critical point of $ \eu{J}_g $, and thus  $ \eu{J}_g'(p)=0 $.
	By the chain rule (\cite[Corollary 1.3.2]{ke}), we have: 
	\begin{equation*}
	\eu{J}_g' (p) =  \eu{I}' (\varphi (p,g))  \circ \dd_1 \varphi (p,g)=0. 
	\end{equation*}
	Since the derivative $ \dd_1 \varphi $ is invertiable at $ p $ by \eqref{itm:sgift.2}, it follows that $ \eu{I}' (\varphi (p,g)) = \zero{\fs{E}'} $. Therefore,  our assumption on $ \eu{I}$ implies that 
	\begin{equation*}
	\varphi(p,g)= \zero{\fs{G}}.
	\end{equation*}
	Now, we prove  by contradiction that $ p $ is the only point for which $\varphi(p,g)= \zero{\fs{G}}$. 
	Let $ e_1 \neq p \in \fs{E} $ be such that  $$ \varphi (e_1,g) =\varphi(p,g) = \zero{\fs{G}}. $$ 
	It follows that $ \eu{J}_g (e_1) = \eu{J}_g(p) = 0 $, from the definition of the function  $ \eu{J}_g(e) = \eu{I} (\varphi(e, g)) $.
	
	Let $ \mathbf{r} > 0 $ be  small enough  such that $ p \notin \Cl{e_1+\bl{r}\bo_{\mt_\fs{E}}} $.
	Without loss of generality, we can suppose that $ e_1 = \zero{\fs{E}}$. 
	For any $ e \in \partial (\zero{\fs{E}}+\bl{r}\bo_{\mt_\fs{E}}) $,  by  \eqref{itm:sgdp.1}, we have $\eu{I}(e) \neq 0 $. Consequently, 
	\begin{equation*}
	0 < m \leq \eu{J}_g (e) \quad \text{for some} \, m.
	\end{equation*}
	Thus, all assumptions of Theorem~\ref{th:mpt} are satisfied. Hence, there exists a sequence $ ( e_n ) \subset\fs{ E} $ such that 
	\begin{equation*}
	\lim_{n \rightarrow \infty} \eu{J}_g (e_n) = c,
	\end{equation*}
	for some $ c > m  $ as characterized by
	\eqref{eq:mpt}. Since $ \eu{J}_g (e_n) $ satisfies the $\pss$-condition at $ c $, it has a convergent subsequent, denoted again by $ (e_n) $,  with the limit $h$. Therefore, $ h $ is a critical point, and thus  $\eu{J}_g'(h)=0 $. Since 
	\begin{equation*}
	\lim_{n\rightarrow \infty} \eu{J}_g (e_n) = \eu{J}_g(h) = c  \geq m >  0,
	\end{equation*}
	it follows that 
	\begin{equation}\label{eq:conti2}
	\varphi(h,g) \neq \zero{\fs{G}}.
	\end{equation}
	By the chain rule (\cite[Corollary 1.3.2]{ke}), we have 
	\begin{equation*}
	\eu{J}'_g (h) =  \eu{I}' (\varphi (h,g) ) \circ \dd_1 \varphi (h,g) =0.
	\end{equation*}
	Since $ \dd_1 \varphi $ is invertible, it follows that $ \eu{J}'_g (h) = \zero{\fs{G}'} $. 
	Thus, \eqref{itm:sgdp.2}  implies that $\varphi(h,g)= \zero{\fs{G}} $,  which  contradicts \eqref{eq:conti2}. 
	To conclude the proof, it is enough to define 
	$
	\eu{K} (g) = p$
	for given $ g \in \fs{F} $. Here, $ p $ is the solution to $ \varphi (p,g)= \zero{\fs{G}} $ obtained as above.  By using the chain rule, we can easily obtain \eqref{eq:der2}.
\end{proof}
\begin{remark}
	The sole distinction between these two implicit function theorems lies in the assumptions concerning the associated functionals. Depending on a particular application we may apply either of them. In the rest of the paper, we can interchangeably use both classes of associated functionals. For convenience, we denote by $ \bl{AU}(\fs{F}, [0,\infty]) $ the set of all functionals  $\eu{I}: \fs{F} \to \rr$ such that either:   
	\begin{enumerate}
		\item $\eu{I}$ is a Keller's $ C^1_c $ satisfying \eqref{itm:sgdp.1} and \eqref{itm:sgdp.2} in Theorem \ref{th:gift2}, or
		\item $\eu{I}$ is a locally Lipchitz function satisfying \eqref{itm:gdp1} and \eqref{itm:gdp2} in Theorem \ref{th:gift1}.
	\end{enumerate}
\end{remark} 
The primary implication  of these theorems is the following global inversion  theorem.
\begin{theorem}\label{th:inv}
	Assume that $\varphi  : \fs{E} \to \fs{F} $ is a Keller's $ C_c^1 $-mapping,  $ \eu{I} \in \bl{AU}(\fs{F}, [0,\infty]) $, and
	the following conditions are satisfied:
	\begin{enumerate}[label={\bf (IFT\arabic*)},ref=IFT\arabic*]
		\itemindent=8pt
		\item \label{ift.1} for any $ f \in \fs{F} $, the function $\eu{J}_f: \fs{E} \to [0,\infty]$  defined by $\eu{J}_f(e) = \eu{I}(f-\va(e))$ satisfies the $\mathrm{ PS}$-condition at any level if $ \eu{I} $
		is a Keller's $ C^1_c $ function. Moreover, it satisfies the Chang $\mathrm{ PS}$-condition at any level if  $ \eu{I} $ is locally Lipschitz. 
		\item \label{ift.2} The derivative $ \va'(e): \fs{E} \to \fs{F} $ is bijective for any $ e \in \fs{E} $.
	\end{enumerate}
	Then, $ \va $ is a global Keller's $ C_c^1 $-diffeomorphism.
\end{theorem}
\begin{proof}
	We assume that $ \eu{I} $ is a Keller's $ C^1_c $-function. The proof remains the same in the case  where $ \eu{I} $ is locally Lipschitz. Define the mapping
	\begin{equation}
	\eu{F}: \fs{E} \times \fs{F} \to \fs{F}, \quad \eu{F}(e,f)=f-\va(e).		
	\end{equation}
	It belongs to the class Keller's $ C^1_c $, and by \eqref{ift.1} the function 
	\begin{equation*}
	\eu{I}(\eu{F}(e,f)) = \eu{I}(f-\va(e))
	\end{equation*}
	satisfies the PS-condition at any level. Moreover, $ \dd_1 \eu{F} = \va' $ is bijective by \eqref{ift.2}. Thus, all the assumptions of the global implicit function are met. Hence, there exists a unique Keller's $ C_c^1 $-mapping $ \eu{K} : \fs{F} \to \fs{E}$ such that 
	\begin{equation*}
	\eu{F} (\eu{K}(f),f) =\zero{\fs{F}}.
	\end{equation*}
	Obviously,   $ \eu{K}(f) =\va^{-1}(f) $ and
	$
	\dd \va^{-1}(f) = \left[\dd \va (\va^{-1}(f))\right]^{-1} .
	$
\end{proof}
\begin{remark}
	In \cite[Theorem 3.1]{k3}, an analogue  of this theorem was proved for Keller's $C_c^2$-tame mappings. Adding,  the tame assumption is necessary to apply the Nash-Moser inverse function theorem. However, we do not require that assumption, and the remaining assumptions of the theorem are slightly less restrictive than those  in  \cite[Theorem 3.1]{k3}.
\end{remark}

\section{An Initial Value Problem}
In this section, we employ the implicit function theorems to establish sufficient conditions for solving the following initial value problem:
\begin{equation}\label{eq:ivp}
\begin{cases}
y'(t)= \upphi ( t, y(t), e ), \quad \forall t \in \bl{I}=[t_0-a,t_0+a] \\
y(t)=f.
\end{cases}
\end{equation}
Here, the values $y(t)$  belong to the \fr space $ \fs{F} $, $ a>0 $, $t_0  \in \rr$, and
$ e $ belongs to the \fr space $ \fs{E} $. In addition, we suppose that $ \upphi : [-1,1]\times \fs{F} \times \fs{E} \to \fs{F} $ is a Keller's $ C_c^1 $-mapping. 
This is a significant and challenging problem beyond the Banach case, and many attempts have been made to provide non-restrictive conditions to solve it.

To solve \eqref{eq:ivp}, we reformulate it as a functional equation. The following spaces will be required.
The space $ \fs{C}_0 \coloneqq \fs{C}^0([-1,1], \fs{F}) $ consists of continuous mapping from $ [-1,1] $ to $ \fs{F} $ is also a \fr space defined by the seminorms:
\begin{equation}
\snorm[\fs{C}_0,n]{u(\cdot)} = \sup_{t \in [-1,1]} \snorm[\fs{F},n]{u(t)}.
\end{equation}
The metric $ \mt_{\fs{C}_0} $, defined as in \eqref{metric}, generates the same topology on $ \fs{C}_0 $.

We denote by $ \fs{C}^k \coloneqq \fs{C}^k ([-1,1],\fs{F})$  the space of all Keller's $ C_c^k $-mappings from $ [-1,1] $ to $ \fs{F} $, $k \geq 1$. This space constitutes a \fr space equipped with the topology
defined by the following family of seminorms:
\begin{equation}
\snorm[	\fs{C}^k,n]{u(\cdot)} = \max_{0 \leq i \leq k} \sup_{t \in [-1,1]} \snorm[\fs{F},n]{u^i(t)}, \;  \text{where} \,u^0(t) \coloneqq u(t).
\end{equation}
This topology is also defined by the metric $ \mt_{\fs{C}} $, as defined  in Equation \eqref{metric}. 

With the aforementioned notations, we now move forward to establish the following result:
\begin{theorem}\label{th:ivp}
	Let $ \eu{I} \in \bl{AU}(\rr \times  \fs{F} \times \fs{E}, [0,\infty]) $, and suppose that the following condition holds:
	\begin{enumerate}[label={\bf (C)},ref=C]
		\item\label{itm:ivp.1} for any $ \psi \in \fs{C}^2 ([-1,1],\fs{F})$, the mapping
		\begin{gather*}
		\eu{J}_{\psi}:  \rr \times  \fs{F} \times \fs{E} \to [0,\infty] \\
		\eu{J}_{\psi} (a,f,e) = \eu{I}\big(\psi'(s) - a\upphi(t_0+as, \psi(s)+f,e)\big),  \quad  \text{for}\, s \in [-1,1].
		\end{gather*}
		satisfies the $\mathrm{ PS}$-condition at any level if $ \eu{I} $
		is a Keller's $ C^1_c $ function. Moreover, it satisfies the Chang $\mathrm{ PS}$-condition at any level if  $ \eu{I} $ is locally Lipchitz. 
	\end{enumerate}
	Then, there exists $ b \in (0,a]$ such that the IVP \eqref{eq:ivp} has a unique solution $ y= y(t;f,e) \in \fs{C}^2 ([-1,1],\fs{F})$  for each
	$ (f,e) \in \fs{F} \times \fs{E} $. In addition, the mapping
	\begin{gather}\label{map:psi}
	\uppsi:(t_0-b,t_0+b) \times \fs{F} \times \fs{E} \to \fs{F}, \quad (t,f,e) \mapsto y(t;f,e)
	\end{gather}
	is of class Keller's $ C_c^1 $.
\end{theorem}
\begin{proof}
	We assume that $ \eu{I} $ is a Keller's $ C^1_c $-function. The proof remains the same in the case  where $ \eu{I} $ is locally Lipchitz.	
	Using the following substitutions: 
	\begin{align}\label{eq:funcy}
	&s= (t-t_0) / a, \nonumber\\
	&\bl{y}(s; t_0,f)= y(t_0+as;f)-f, \quad \forall s \in \bl{J}=[-1,1],
	\end{align}
	we will reformulate the IVP \eqref{eq:ivp} into a more convenient form.
	Thus, the IVP \eqref{eq:ivp} transforms into the following problem:
	\begin{equation}\label{eq:ivp2}
	\begin{cases}
	\bl{y}'(s)= a\upphi ( t_0+as, \bl{y}(s)+f,e), \quad \forall s \in \bl{J}, \\
	\bl{y}'(0) = \zero{\fs{F}}.
	\end{cases}
	\end{equation}
	Let $ \fs{D} \coloneqq \Set{\va \in \fs{C}^2 (\bl{J},\fs{F})}{\va(0) = \zero{\fs{F}}} $, which is a closed linear subspace of $ \fs{C}^2 (\bl{J},\fs{F})$, and therefore is a \fr space.
	We cast \eqref{eq:ivp2} as a functional equation by introducing the Keller's $ C_c^1 $-mapping: 
	\begin{gather}\label{eq:opr}
	\eu{F}: \fs{D} \times \rr \times  \fs{F} \times \fs{E} \to \fs{C}^1 \\
	\eu{F}(\bl{y},a,f,e) \coloneqq \bl{y}'(s)-a \upphi(t_0+as,\bl{y}(s)+f,e), \quad  s \in \bl{J}. \nonumber
	\end{gather}
	By \eqref{itm:ivp.1}, the mapping 
	\begin{equation}
	\eu{I}\big(\eu{F}(\bl{y},a,f,e) \big) = \eu{I} \big( \bl{y}'(s)-a \upphi(t_0+as,\bl{y}(s)+f,e)\big)
	\end{equation} 
	satisfies the PS-condition at all levels. Differentiating  $\eu{F}$ with respect to $ \bl{y} $ yields  $ \eu{F}_{\bl{y}}(\cdot)\bl{y}=\bl{y}' $. Moreover, for any $ \va \in \fs{C}^1 $, there exists a unique $ \bl{y} \in \fs{D}$,  given by 
	\begin{equation*}
	\bl{y}(s) = \dit_{0}^{s}\va(t) \dt{t}
	\end{equation*}
	such that $ \bl{y}' = \va$. Therefore, the mapping
	$ \eu{F}'_{\bl{y}} : \fs{D} \to \fs{C}^1$ is bijective. Thus, all the assumptions of the implicit function theorem \ref{th:gift2} are fulfilled, implying that for a given $ a>0, e \in \fs{E}, $ and $ f \in \fs{F} $ the functional equation
	\begin{equation}
	\eu{F}(\bl{y},a,f,e) = \zero{\fs{C}}
	\end{equation}
	has a unique solution $ {y} \in \fs{D} $ that solves the IVP \eqref{eq:ivp}. In addition, the mapping
	\begin{gather} \label{gath:mapupvarphi}
	\upvarphi	:\rr \times \fs{F} \times \fs{E} \to \fs{D}, \quad (a,f,e) \mapsto {y}
	\end{gather}
	is of class Keller's $ C_c^1 $.
	
	Next, we prove that the mapping $ \uppsi $ defined by Equation \eqref{map:psi}, is Keller's $ C_c^1 $. 
	Consider the solution $ y $ and $ \vv \coloneqq (t_0-b,t_0+b) \times \fs{F}$, where $ 0< b \leq a $. We will establish the continuity of the partial derivatives $ {y}_t(t,f) $ and $ {y}_f(t,f) $ on
	$\vv$. Subsequently, by \cite[Proposition II.2.6]{ke} it follows  that $ \uppsi $ is a Keller's $ C_c^1 $-mapping.
	
	First, we prove that:
	\begin{quote} \label{quote}
		(1)	\quad $(t,f) \mapsto {y}(t,f)$ is continuous on $ \vv $. 
	\end{quote}
	The continuity of the mapping $ \upvarphi $, as defined in \eqref{gath:mapupvarphi}, implies that for any given $ \epsilon >0 $:
	\begin{equation}
	\mt_{\fs{C}} \big(  {\bl{y}(t_0+\mathtt{t}, f+\mathtt{f}) , \bl{y}(t_0,f)} \big) < \epsilon
	\end{equation}
	if $ \abs{t - \mathtt{t}} $ and $ \mt_{\fs{F}}({f,\mathtt{f}})$ are sufficiently small. 
	This means that
	\begin{equation}
	\forall s \in \bl{J}, \quad	\mt_{\fs{C}_0} \big(
	\bl{y}(s;t_0+\mathtt{t},f+\mathtt{f}), \bl{y}(s;t_0,f)
	\big) < \epsilon.
	\end{equation}
	Thus,  by the definition of  $ \bl{y} $  in \eqref{eq:funcy},  the mapping $ (t,f) \mapsto {y}(t,f) $ is continuous on $ \vv $.
	Now, by employing $ (1) $ and the IVP \eqref{eq:ivp}, we can show that:
	\begin{quote} 
		(2)	\quad $(t,f) \mapsto {y}_t(t,f)$ is continuous on $ \vv $. 
	\end{quote}
	Since $ \upvarphi $ is a Keller's $ C_c^1 $-mapping, the continuity of the partial derivative
	\begin{equation*}
	\mathtt{D}\coloneqq \bl{y}_f : \fs{F} \to \fs{D}
	\end{equation*}
	follows. This, in turn, implies that for any given $ \epsilon >0 $:
	\begin{equation}\label{eq:34}
	\forall e \in \fs{F}, \quad	\mt_{\fs{C}} \big(  {\dd(t_0+\mathtt{t}, f+\mathtt{f})e , \dd(t_0,f)}e \big) < \epsilon
	\end{equation}
	if $ \abs{t - \mathtt{t}} $ and $ \mt_{\fs{F}}({f,\mathtt{f}})$ are sufficiently small. 
	Furthermore, it follow from \eqref{eq:34} that, for $ s \in \bl{J} $
	\begin{equation}\label{eq:last}
	\forall e \in \fs{F}, \quad \mt_{\fs{F}}\left( \bl{y}_f(s;t_0+\mathtt{t}, f + \mathtt{f})e , \bl{y}_f(s;t_0, f)e\right) <\epsilon
	\end{equation}
	if $ \abs{t - \mathtt{t}} $ and $ \mt_{\fs{F}}({f,\mathtt{f}})$ are sufficiently small. Thereby,
	by the definition of  $ \bl{y} $  in \eqref{eq:funcy} and \eqref{eq:last} we have
	\begin{quote}
		(3) \quad $(t,f) \mapsto {y}_f(t,f)$ is continuous on $ \vv $. 
	\end{quote}
\end{proof}
\begin{remark}
	In \cite{vanthera}, the local existence and uniqueness of the solution to the IVP \eqref{eq:ivp} have been established for \fr spaces known as standard. In this context, the data function $ \upphi $ is assumed to be Hadamard differentiable, but this alone does not establish  uniqueness. To prove uniqueness, it is necessarily for   $ \upphi $ to be continuously differentiable and satisfy a specific condition. This result has be modified in \cite{ivan}, by dropping the assumption that the spaces are standard and  assuming that
	the mapping $ \upphi $ is G\^{a}teaux differentiable. However,  in this setting, a unique solution is not guaranteed in general.
	The problem also has been studied in \cite{p}, assuming that spaces exhibit the properties $ (S_{\Omega})_t $ and $ (DN) $, and that mappings are twice continuously differentiable.
	
	In our theorem, \fr spaces are supposed to be arbitrary, and mappings are continuously differentiable. However, the Palais-Smale condition, which plays a central role, may seem restrictive compared to the other results.
\end{remark}

\section{Lagrange Multipliers}
In this section, we prove a Lagrange multiplier theorem and investigate the applicability of the Nehari method in our context. We begin by recalling the definitions of submanifolds, regular points, and infinitesimally surjective maps.

Suppose $ \fs{F}_1 $ is a closed subset of the \fr space $ \fs{F} $ that splits it. Let $ \fs{F}_2 $
be one of its complements, i.e., $ \fs{F} = \fs{F}_1 \oplus \fs{F}_2 $. A subset $ \fs{S} $ of  $ \fs{F} $ is called a \fr submanifold modeled on $ \fs{F}_1 $ of class Keller's $ C_c^r $, $ r\geq 1 $, if for any $ p \in \fs{S} $ there
exists a Keller's $ C^r_c $-diffeomorphism $ \phi : U_{\phi} \to V_{\phi} $, with $ U_{\phi} \ni p \opn  \fs{F}$ and  $ V_{\phi}=W_{\phi} \times O_{\phi} \opn \fs{F}_1 \times \fs{F}_2 =\fs{F} $,
such that
\begin{equation*}
\phi (\fs{S} \cap U_{\phi}) = W_{\phi} \times \set{\zero{\fs{F}_2}}.
\end{equation*} 
Then $\fs{S}$ is a Keller's $C_c^r$-\fr manifold modeled on $\fs{F}_1$, with the maximal Keller's $C_c^r$-atlas including the mappings $\phi|_{{U_{\phi} \cap \fs{S}}}: {U_{\phi} \cap \fs{S}} \to {V_{\phi} \cap \fs{S}}$ for all $\phi$ as described above.

When dealing with submanifolds, it is often more practical to work with tangent maps rather than differentials. To revisit this concept, consider a mapping  $ \va: U \opn \fs{E} \to V \subset \fs{F} $ of class Keller's $ C^{r} $, $ r\geq 1 $. The tangent map of $ \va $ at $ u \in U $ is defined by
\begin{gather*}
T_u\va: TU \to TV, \quad T_u\va(u,e) = (\va(u),\va'(u)e),
\end{gather*}
where $ TU= U \times \fs{E} $ and $ TV= V \times \fs{F} $. Let $ \Pr_{U} $ and $ \Pr_{V} $
be the projections of $ TU $ and $ TV $ onto $ U $  and $V$, respectively. Then, the diagram
\begin{equation*}
\begin{CD}
TU @>T\varphi>> TV\\
@VV\Pr_UV @VV\Pr_{V}V\\
U @>\va>> V
\end{CD}
\end{equation*}
is commutative.
Let   $ \va: \fs{E} \to \fs{F} $ be a Keller's $ C_c^r $-mapping, $ r\geq1 $. We call
$ \va $ infinitesimally surjective at $ e \in \fs{E}$ if the tangent map $ T_{e}\va $ is surjective. 

A point $ f \in \fs{F} $ is called a regular value of $ \va $ if $ \va  $ is infinitesimally surjective  at each $ e \in \va^{-1}(\set{f} )$, and the tangent map  $ T_e\va $ has the split kernel.
If $ T_e\va $ is not surjective, we call $ e \in \fs{E} $ a critical point of $ \va $.
\begin{theorem}\label{th:gst}
	Suppose that a Keller's $ C_c^1 $-mapping $  \va: \fs{E} \to \fs{F}$ is infinitesimally surjective  at $ e_0 $, and $\ker \va'(e_0)$ splits $ \fs{E} $. Let $ \fs{E}_1 $ be the closed complement of 
	$\fs{E}_2 \coloneqq \ker \va'(e_0)$. Suppose that  $ \eu{I} \in \bl{AU}(\fs{F} \oplus \fs{E}_2, [0,\infty]) $, and
	the following condition is satisfied:
	\begin{enumerate}[label={\bf (S)},ref=S]
		\item \label{s.1} for any $ v \in \fs{F}\oplus \fs{E}_2 $, the function 
		\begin{equation*}
		\eu{J}_v: \fs{E} = \fs{E}_1 \oplus \fs{E}_2 \to [0,\infty], \quad \eu{J}_v(u_1,u_2)= \eu{I}\big(v-(\va(u_1,u_2),u_2)\big)
		\end{equation*}
		satisfies the $\mathrm{ PS}$-condition at any level if $ \eu{I} $
		is a Keller's $ C^1_c $ function. Moreover, it satisfies the Chang $\mathrm{ PS}$-condition at any level if  $ \eu{I} $ is locally Lipchitz. 
	\end{enumerate}
	Then, there exists a Keller's $ C_c^1 $-diffeomorphism $ \phi:  \fs{F} \oplus \fs{E}_2 \to \fs{E}
	$ such that the following diagram commutes:
	\begin{equation*}
	\xymatrix{
		\mathbb{F}\oplus \mathbb{E}_2 \ar[r]^{\phi} \ar[dr]_{\mathrm{Pr}_{\fs{F}}} &  \mathbb{E} \ar[d]^{\varphi} \\
		& \fs{F}
	}
	\end{equation*}	
	Here, $ \mathrm{Pr}_{\fs{F}} $ is the projection onto $ \fs{F} $.  Furthermore, the restriction of $\dd \phi(f,e)$ to
	$ \fs{F} \times \set{\zero{\fs{E}}} $ is an isomorphism for all $ (f,e)\in \fs{F}\oplus \fs{E}_2 $.
\end{theorem}
\begin{proof}
	Again, we assume that $ \eu{I} $ is a Keller's $ C^1_c $-mapping. 
	By the open mapping theorem, $ \va'(e_0)|_{\fs{E}_1} : \fs{E}_1 \to \fs{F}$ is continuous linear isomorphism.
	Define the mapping:
	\begin{equation*}
	\upvarphi : \fs{E} \to \fs{F}\oplus \fs{E}_2, \quad \upvarphi(u)=(\va(u_1,u_2),u_2),\, \text{where}\, u= (u_1,u_2) \in \fs{E}_1 \oplus \fs{E}_2.
	\end{equation*}
	Hence,
	\begin{equation*}
	\dd \upvarphi (u)(e_1,e_2)=
	\stackrel{\mbox{$A$}}{
		\begin{bmatrix}
		\dd_1 \va(u) & \dd_2 \va(u)  \\
		\zero{\fs{F}} & \Id_{\fs{E}_2} 
		\end{bmatrix}}
	\begin{bmatrix}
	e_1  \\
	e_2 
	\end{bmatrix}, \quad \text{for all}\,  u=(u_1,u_2) \in \fs{E},\,e_1 \in\fs{E}_1, \,e_2 \in\fs{E}_2.
	\end{equation*}
	However, $ \dd_2 \va(e_0) = \dd\va(e_0)|_{\fs{E}_2} = \zero{\fs{F}} $, since $ \fs{E}_2 = \ker \dd \va(e_0) $.
	Consequently, $ A $ is block diagonal, and thus $ \dd \upvarphi(e_0): \fs{E} \to \fs{F} \oplus \fs{E}_2 $ is a continuous linear isomorphism. By \eqref{s.1}, for any $ v \in \fs{F}\oplus \fs{E}_2 $, the function
	\begin{equation*}
	\ \eu{J}_v((u_1,u_2)) = \eu{I}(v-\upvarphi(u_1,u_2))
	\end{equation*}
	satisfies the PS condition at any level. Therefore, all the hypothesis of the global inversion theorem \eqref{th:inv} are satisfied for $ \upvarphi $, and hence, $ \upvarphi : \fs{E} \to  \fs{F}\oplus \fs{E}_2$ is a diffeomorphism of class 
	Keller's $ C^1_c $. Let $ \phi \coloneqq \upvarphi^{-1} $, which is also a Keller's $ C^1_c $-diffeomorphism. Then
	\begin{equation*}
	\forall (f,e)\in \fs{F}\oplus \fs{E}_2, \quad	(f,e)=(\upvarphi\circ\phi)(f,e)=(\va(\phi(f,e)), e).
	\end{equation*}
	Thus, $ \upvarphi\circ\phi(f,e) = f = \mathrm{Pr}_{\fs{F}}(f,e) $. Furthermore, the mapping $ \dd \phi (\va(e_0),e_{02})$, where $ e_0=(e_{01},e_{02}) $, exhibits a block diagonal structure, since $ \dd \upvarphi(e_0) $ is block diagonal.
	This, along with the fact that $ \dd \upvarphi(e) $ is an isomorphism for all $ e \in \fs{E} $, implies that $ \dd_1 \va(e): \fs{E}_1 \to \fs{E} $ is also an isomorphism. This means that 
	$ \dd \upvarphi(e)|_{\fs{E}_1 \times \set{\zero{\fs{E}}}}: \fs{E}_1 \times \set{\zero{\fs{E}}} \to \fs{F} \times \set{\zero{\fs{E}_2}}  $ is an isomorphism. Therefore, 
	\begin{equation*}
	\dd \phi((f,e))|_{\fs{F} \times \set{\zero{\fs{E}}}}: \fs{F} \times \set{\zero{\fs{E}}} \to \fs{E}_1 \times \set{\zero{\fs{E}_2}} 
	\end{equation*}
	is an isomorphism for all $ (f,e)\in \fs{F}\oplus \fs{E}_2 $.
\end{proof}
\begin{corollary}\label{cor:1}
	Assume that $ y \in \fs{F} $ is a regular value of $ \va $, and let $ \fs{S}=\va^{-1}(y) $. If $ \va $ satisfies condition \eqref{s.1}  at some point of $ \fs{S} $, then the preimage $ \fs{S}=\va^{-1}(y) $ is a  Keller's $ C^1_c $-\fr submanifold. The tangent space at any point $ x \in \fs{S} $ is given by $ T_x\fs{S} = \ker \va'(x)$.
\end{corollary}
\begin{proof}
	We consider $ \fs{S} $  in a local neighborhood around a given point $ x $. By Theorem \ref{th:gst}, there exists a Keller's
	$ C^1_c $-diffeomorphism $ \phi $ such that
	\begin{equation*}
	\forall h \in \fs{E}, \quad \va \circ \phi^{-1} (h)= \va'(x)h+y.
	\end{equation*} 
	Consequently, the solution set of the equation
	\begin{equation*}
	\va(z)=y
	\end{equation*}
	in a neighborhood of $ z=x $ corresponds to the solution set of the equation 
	\begin{equation*}
	\va'(x)h=\zero{\fs{F}}
	\end{equation*}
	in a neighborhood of $ \zero{\fs{F}} $. Therefore, locally, $ \fs{S} $ looks like $\ker \va'(x)$, which splits in $\fs{E}$.
	
	To determine the tangent space $ T_x\fs{S} $, consider a Keller's $ C^1_c $-curve $ c(t) $ in $ \fs{S} $ with $ c(0)=x $. Since $ \va (c(t)) = \zero{\fs{F}} $, it follows that $ \va'(x)c'(0)=\zero{\fs{F}} $. Therefore, $ T_x\fs{S} \subset \ker \va'(x) $.
	
	Conversely, suppose $ \va'(x)w = \zero{\fs{F}} $, and let $ x(t) = \phi^{-1}(tw) $. Then, $ x(t) \subset \fs{S} $, and $ x'(0)=w $. Hence, $ T_x\fs{S} = \ker \va'(x) $.
\end{proof}
Moving forward, we define transversality. Let $ \fs{S} $ be a Keller's $ C_c^1 $-submanifold of $\fs{F}$ modeled on $ \fs{F}_1 $, where $ \fs{F}_1 $ is a complemented closed subset of $ \fs{F} $, and $ \fs{F}_2 $ is one of its complements, i.e., $ \fs{F} = \fs{F}_1 \oplus \fs{F}_2 $.
Let $ \va : \fs{E} \to \fs{F} $ be a Keller's $ C^1_c $-mapping. We say $ \va $ is transversal to $ \fs{S} $ and denote it as $ \va \pitchfork \fs{S} $ if the following condition is met:
\begin{enumerate}[label={\bf (T)},ref=T]
	\item \label{c:t}	consider $ x \in \fs{E} $ such that $ \va(x) \in \fs{S} $, and let $ (U,\phi) $ be a chart at
	$ \va(x) $ such that $ \phi: U \to U_1 \times U_2 $ is an isomorphism, satisfying
	\begin{equation*}
	\phi (\va(x)) =(\zero{\fs{F}_1},\zero{\fs{F}_2}) \quad \text{and} \quad \phi (\fs{S} \cap U) = U_1 \times \zero{\fs{F}_2}.
	\end{equation*}
	Then, there exist an open neighborhood $ \bl{U} $ of $ x $ such that the composite map
	\begin{equation*}
	\bl{U} \xlongrightarrow{\va} U \xlongrightarrow{\phi} U_1 \times U_2 \xlongrightarrow{\Pr_{2}} U_2
	\end{equation*}
	is a infinitesimally surjective, and  $\ker \va'(x)$ splits $ \fs{E} $. 
\end{enumerate}
With the above notations, we proceed to prove the following result:
\begin{corollary} \label{cor:2}
	If $ \va \pitchfork \fs{S} $ and condition \eqref{s.1} is fulfilled at some point $ s \in \va^{-1}(\fs{S}) $, then $ \va^{-1}(\fs{S}) $ is  a Keller's $ C_c^1 $-\fr submanifold of $ \fs{E} $.
\end{corollary}
\begin{proof}
	The transversality condition \eqref{c:t} implies that for $ x \in \bl{U} \cap \va^{-1}(\fs{S}) $, we have
	\begin{equation*}
	T_x ( \mathrm{Pr}_2 \circ \phi \circ \va|_{\bl{U}} ) =  \mathrm{Pr}_2 \circ T_{\va(x)} \circ T_x\va ,
	\end{equation*}
	and $ T_x (\phi \circ \va)(\fs{E}) + \fs{F}_1  = \fs{F}_1 \oplus \fs{F}_2$.  Therefore, $ T_x ( \mathrm{Pr}_2 \circ \phi \circ \va|_{\bl{U}} ): \fs{E} \to \fs{F}_2 $ is surjective, and its kernel is $ (T_x\va)^{-1}(T_{\va(x)}\fs{S}) $. This is due to the fact that
	$ \ker \Pr_{2} = \fs{F}_1 $, and we have
	\begin{equation*}
	(T_{\va(x)}\phi)^{-1}(\fs{F}_1) = T_{\va(x)}\fs{S}.
	\end{equation*}
	Consequently, it splits in $ \fs{E} $, implying that
	$ \zero{\fs{F}_2} $ is a regular value of the composition $$ 	(\mathrm{Pr}_2 \circ \phi \circ  \va|_{\bl{U}}): \bl{U}  \to \fs{F}_2.$$
	Furthermore, 
	\begin{equation*}
	( \mathrm{Pr}_2 \circ \phi \circ  \va|_{\bl{U}} )^{-1} (\zero{\fs{F}_2})= \va^{-1} (\fs{S} \cap \bl{U}).
	\end{equation*}
	Therefore, by
	Corollary \ref{cor:1},
	\begin{equation*}
	\va^{-1} (\fs{S} \cap U)
	\end{equation*}
	is a Keller's $ C_c^1 $-\fr submanifold of $ \fs{E} $. Its tangent space at
	$ x \in \bl{U} $ equals 
	\begin{equation*}
	\ker \big(T_x (\mathrm{Pr}_{2} \circ \phi \circ \va|_{\bl{U}}) \big) = (T_x \va )^{-1}(T_{\va(x)}\fs{S}).
	\end{equation*}
\end{proof}
Next, we will prove a Lagrange multiplier theorem. 
Let $ \phi: \fs{E} \to \rr $ be a Keller's $ C^1_c $-mapping and $ \bl{S} $ a submanifold of $ \fs{E} $.
A point $ s \in \bl{S} $ is a critical point of $ \phi|_{\bl{S}} $ if and only if $ \langle \phi'(s),v \rangle = 0 $ for all
$ v \in T_s\bl{S} $. By definition of a tangent space, this means that for every Keller's $ C^1_c $-mapping $ \gamma : (-\epsilon,\epsilon) \to \fs{E} $ such that $ \gamma(t) \in \bl{S} $ for all $ t \in (\-\epsilon,\epsilon) $, $ \gamma(0)=s $, and $ \gamma'(0)  $ exists, we have
\begin{equation*}
\derat0  \va(\gamma (t)) = 0.
\end{equation*}
If $ s \in \text{Int}\, \bl{S} $, then $ s $ is a usual critical point of $ \phi $, which is called a free critical point of $ \phi $.

Let $\eu{I}: \bl{S} \hookrightarrow \fs{E}  $ be the canonical inclusion so that $ \phi|_{\bl{S}} = \phi \circ \eu{I} $. Then the chain rule implies that $ T_s(\phi|_{\bl{S}}) = T_s \phi \circ T_s \eu{I}$, thereby
$ s $ is a critical point of $ \phi|_{\bl{S}} $ if and only if $ T_s\phi|_{T_s\bl{S}}=0 $.

Let $  \va: \fs{E} \to \fs{F}$ be a Keller's $ C_c^1 $-mapping, $ y $ be a regular value of $ \va $, and $\fs{S}=\va^{-1}(y)$. Therefore, $ \va $ is infinitesimally surjective  at $ y $, with  $ \fs{E}_2 \coloneqq \ker \va'(y)$ that splits $ \fs{E} $ such that $ \fs{E}_2 \oplus \fs{E}_1 = \fs{E} $.  Suppose that there exist  $ \eu{I} \in \bl{AU}(\fs{F} \oplus \fs{E}_2, [0,\infty]) $  such that condition \eqref{s.1} is satisfied. 
Thus, by Corollary \ref{cor:1}, $\va^{-1}(y)$ is a Keller's $ C^1_c $-submanifold of $ \fs{E} $. 

With the above notations, in the following theorem for a real-valued mapping $ \upphi $ on $\fs{E}$, we give necessary and sufficient conditions for a point in $ \fs{S} $ to be a critical point of $ \upphi|_{\fs{S}} $.
\begin{theorem}\label{th:lm}
	Let $ \upphi: \fs{E} \to \rr $ be a Keller's $ C^1_c $-mapping satisfying condition \eqref{s.1} at some point of $ \fs{S} $. A point $ s \in \fs{S} $ is a critical point of
	$ \upphi|_{\fs{S}} $ if and only if there exist $ \mu \in \fs{F}' $ such that $ s $ is a critical point of
	$ \upphi - \mu \circ \va $. If $ \va $ is surjective, then $ 
	\mu $ is unique.
\end{theorem}
\begin{proof}
	{\bf Sufficiency.} Suppose that such a $ \mu $ exists, and $ s $ is a critical point of  
	$ \upphi - \mu \circ \va $. In terms of tangent maps we have $ T_s\upphi = \mu \circ T_s \va $ and
	$ T_s \fs{S} = \ker T_s{\va} $. Therefore, if $ \eu{I}: \fs{S} \hookrightarrow \fs{E} $ is the canonical inclusion, then 
	\begin{equation*}
	\forall v \in T_s\fs{S}, \quad (\mu \circ T_s \va \circ T_s \eu{I})(v) = \mu (T_s \va(v)) =0,
	\end{equation*}
	implying that $ 0 = (\mu \circ T_s \va) |_{T_n\fs{S}} = T_s\upphi|_{T_s\fs{S}} $.
	
	{\bf Necessity.} Suppose that $ s  $ is a critical point of $ \upphi|_{\fs{S}} $, that is $ T_s\upphi|_{T_s\fs{S}} =0 $.
	It follows from Theorem \ref{th:gst} that there exist  charts
	$ \phi: U \ni s  \to U_1 \times V_1 \subset \fs{E} \times \fs{F}$
	and $ \psi: V \ni y \to U_1 $ with $ \va(U) \subset V $ satisfying
	\begin{equation*}
	\phi(s) = (\zero{\fs{E}}, \zero{\fs{F}}),\, \psi(y) = \zero{\fs{E}},\, \text{and}\,\, \phi (U \cap \fs{S}) = \set{\zero{\fs{E}}} \times V_1,
	\end{equation*}
	such that
	\begin{equation*}
	\forall (u,v) \in U_1 \times V_1, \quad (\psi \circ \va \circ \phi^{-1})(u,v)=u.
	\end{equation*}
	It follows from  $ T_s\upphi|_{T_s\fs{S}} =0 $ that 
	\begin{equation*}
	\forall f \in \fs{ F}, \quad \dd_2 \upphi \big( \phi^{-1}(\zero{\fs{E}}, \zero{\fs{F}})\big)f=0. 
	\end{equation*}
	Now, set $ \upmu = \dd_1 \upphi \big( \phi^{-1}(\zero{\fs{E}}, \zero{\fs{F}})\big) \in \fs{E}'$. Thereby,
	for $ (e,f) \in \fs{E} \times \fs{F} $, we have
	\begin{equation}
	\dd \upphi \big( \phi^{-1}(\zero{\fs{E}}, \zero{\fs{F}})\big)(e,f) =\upmu(e) = 
	\big(\upmu \circ \dd(\psi \circ \va \circ \phi^{-1}) \big)(\zero{\fs{E}}, \zero{\fs{F}})(e,f),	
	\end{equation}
	which implies 
	\begin{equation}\label{eq:100}
	\dd \upphi \big( \phi^{-1}(\zero{\fs{E}}, \zero{\fs{F}})\big) =
	\big(\upmu \circ \dd(\psi \circ \va \circ \phi^{-1}) \big)(\zero{\fs{E}}, \zero{\fs{F}}).	
	\end{equation}
	Now, let $ \mu = \upmu \circ \psi'(y) \in \fs{F}'$, hence by composing \ref{eq:100} with $ T_s\phi $ we obtain
	$ T_s \upphi = \mu \circ T_s\va $, in another words $ s $ is a critical point of $ \upphi - \mu \circ \va $.
	
	Now, suppose that $ \va $ is surjective, we prove the uniqueness of $ \mu $. Suppose there exists another $ \mu_1  \neq \mu \in \fs{F}' $ such that $ T_s \upphi = \mu \circ T_s\va = \mu_1 \circ T_s\va $. Let $ f \in \fs{F} $ be such that  
	$\langle \mu, f \rangle  \neq \langle \mu_1, f \rangle $. If, $ f = \va(e) $ for $ e \in \fs{E} $, then
	\begin{equation*}
	\langle T_s \upphi, e \rangle	=\langle \mu, \va(e) \rangle  \neq \langle \mu_1, \va(e) \rangle = \langle T_s \upphi, e \rangle,
	\end{equation*}
	which is a contradiction. 
\end{proof}
Now, suppose that $ \fs{S}  $ is a submanifold of $\fs{F}$ modeled on $ \fs{F}_1 $ and $ \va \pitchfork \fs{S} $. If condition \eqref{s.1} is fulfilled at some point $ s \in \va^{-1}(\fs{S}) $,
then by Corollary \ref{cor:2}, $ \bl{S}=\va^{-1}(\fs{S}) $ is  a Keller's $ C_c^1 $-\fr submanifold of $ \fs{E} $.
Consider a Keller's $ C_c^1 $-mapping $ \upphi : \fs{E} \to \rr $ and a point $ s  $ in $ \bl{S} $. 
As $ T_{\va(x)} \fs{S}$ is complemented in $ T_{\va(s)}\fs{F} $, we have 
\begin{equation*}
T_{\va(s)}\fs{F} =  T_{\va(s)}\fs{S} \oplus \bl{F}_{\va(s)},
\end{equation*}
where $ \bl{F}_{\va(s)} $ is one of its complements. 
With this setup, we proceed to characterize critical points of $ \upphi|_{\bl{S}} $ in the next theorem.
\begin{theorem}	Suppose that $ \upphi: \fs{E} \to \rr $ satisfies condition \eqref{s.1} at some point of $ \bl{S} $.
	A point $ s \in \bl{S} $ is a critical point of $ \upphi|_{\bl{S}} $ if and only if there exists 
	$ \mu \in \bl{F}_{\va(s)}' $  such that $ T_s\upphi = \mu \circ \Pr \circ T_s\va $, where $ \Pr: T_{\va(s)}\fs{F} \to  \bl{F}_{\va(s)}$ is the projection. 
\end{theorem}
\begin{proof}
	By Corollary \ref{cor:2},  there exist charts 
	\begin{equation*}
	\phi : U \to U_1 \times U_2  \subset \fs{E}_1 \times \fs{E}_2\,\, \text{and}\,\, 
	\psi : V \to U_1 \times V_1 \subset \fs{E}_1 \times \fs{F} 
	\end{equation*}
	with $ \va(U) \subset V$, satisfying
	\begin{equation*}
	\va(s) = (\zero{\fs{E}_1}, \zero{\fs{E}_2}),\, \va (U \cap \bl{S})=\set{\zero{\fs{E}_1}}\times \zero{\fs{E}_2},\,\,
	\text{and}\,\, \psi (\va (s)) = (\zero{\fs{E}_1}, \zero{ \fs{F}}),
	\end{equation*}
	such that
	\begin{equation*}
	\forall (u,v) \in U_1 \times V_1, \quad (\psi \circ \va \circ \phi^{-1})(u,v) = (u, \lambda(u,v)),
	\end{equation*}
	where $ \lambda : U_1 \times V_1 \to V_1 $ is a Keller's $ C_c^1 $-mapping such that \dd $(\lambda \circ\psi^{-1})(s)=0$. Consider the mapping $ \uppi \coloneqq \Pr_{2} \circ \psi \circ \va \circ \phi^{-1} $,
	where $ \Pr_{2}: \fs{E}_1 \times \fs{F} \to \fs{E}_1 $ is the canonical projection. Applying Theorem \ref{th:lm} to
	$ \uppi $ implies that $ (\zero{\fs{E}_1}, \zero{\fs{E}_2}) $ is a critical point
	of $ \upphi|_{\zero{\fs{E}_1} \times U_2} $ if and only if there exists $ \upmu \in \fs{E}_1' $ such that
	\begin{equation}\label{eq:lm1}
	\dd (\upphi \circ \phi^{-1})(\zero{\fs{E}_1},\zero{\fs{E}_2}) = \upmu \circ \mathrm{Pr}_{2} \circ \dd (\psi \circ \va \circ \phi^{-1}) (\zero{\fs{E}_1},\zero{\fs{E}_2}).
	\end{equation}
	Here, $ \dd (\upphi \circ \phi^{-1})(\zero{\fs{E}_1},\zero{\fs{E}_2}) $ represents the derivative of $ \upphi $ composed with the inverse of $ \phi $ at the point $ (\zero{\fs{E}_1},\zero{\fs{E}2}) $, and $ \upmu \circ \mathrm{Pr}_{2} \circ \dd (\psi \circ \va \circ \phi^{-1}) (\zero{\fs{E}_1},\zero{\fs{E}_2}) $ involves the derivative of $ \psi \circ \va \circ \phi^{-1} $ at the same point, projected onto the second component of the target space.
	
	If we let $ \mu = \upmu \circ T_{\va(s)}\psi \in \bl{F}_{\va(s)}'  $ and compose $ T_s\upphi $ on the the right to Equation \eqref{eq:lm1}, and  define the projection operator $ \Pr $ as:
	\begin{equation*}
	\Pr = (T_{\va(s)}\psi)^{-1}|_{\bl{F}_{\va(s)}} \circ \mathrm{Pr}_{2} \circ T_{\va(s)}\psi : T_{\va(s)}\fs{F} \to \bl{F}_{\va(s)},
	\end{equation*}
	then we obtain $ T_s\upphi = \mu \circ \Pr \circ T_s\va $.
\end{proof}
Next, let's consider a special case when $ \fs{F} = \rr $. Suppose $ \upphi: \fs{E} \to \rr $ is a Keller's $ C^1_c $-mapping and $ \bl{S} $ is a submanifold of $ \fs{E} $. If a point $ s \in \bl{S} $ is a critical point of $ \upphi $, then it is also a critical point of $ \upphi|_{\bl{S}} $. However, the converse is not true in general. To identify critical points of $ \phi $, we can search for a submanifold of $ \fs{E} $ such that both $ \upphi $ and its restriction to the submanifold share the same critical points. This approach is known as the Nehari method.

Consider the following subset of $ \fs{E} $:
\begin{equation*}
\mathcal{N}\coloneqq \Set{e \in \fs{E}\setminus \set{\zero{\fs{E}}}}{\langle \upphi'(e),e \rangle = 0}.
\end{equation*}
This set known as a Nehari manifold, although it is not a manifold in general. To turn it into a submanifold of $ \fs{E} $, we impose the following conditions:
\begin{enumerate}[label={\bf (N\arabic*)},ref=N\arabic*]
	\item \label{eq:1000} there exists an open neighborhood $ U $ of zero such that $ U \cap \mathcal{N} =\emptyset$, 
	\item \label{eq:10001} the function $ \upphi $ is of class Keller's $ C^2_c $, and
	\begin{equation}
	\forall x \in \mathcal{N}, \quad \langle \upphi''(x)x,x\rangle \neq 0.
	\end{equation}	
\end{enumerate}
Define the Keller's $ C^1_c $-mapping $ \va(e)= \langle \upphi'(e),e\rangle $ on $ \fs{E} $. 
Then, $ \mathcal{N} = \va^{-1}(0)\setminus\set{\zero{\fs{E}}} $, and for any $ x \in \mathcal{N} $, we have
\begin{equation}\label{eq:2000}
\langle \va'(x),x \rangle = 	\langle\upphi''(x)x,x \rangle +	\langle \upphi'(x),x\rangle = 	\langle \upphi''(x)x,x \rangle \neq 0.
\end{equation}
Thus, for all $ x\in \mathcal{N} $, we have $ \va'(x)  \neq 0 $, which along with	\eqref{eq:1000} implies that
$ \mathcal{N} $ is a Keller's $ C^1_c $-submanifold of $ \fs{E} $.

Suppose that $ \upphi: \fs{E} \to \rr $ satisfies condition \eqref{s.1} at some point of $ \mathcal{N} $.
If $ s $ is a critical point of $ \upphi|_{\mathcal{N}} $, then by Theorem \eqref{th:lm}, there exists
$ \mu \in \rr $ such that:
\begin{equation*}
\upphi'(s) = \mu \va'(s).
\end{equation*}
Therefore, $ \langle \upphi'(s),s \rangle = \mu \langle \va'(s),s \rangle =\va(s)=0$. Thus, it follows from
\eqref{eq:2000} that $ \mu =0 $, and hence $ \upphi'(s)=0 $. Therefore, we have proved the following theorem.
\begin{theorem} \label{th:ls}
	Let $ \upphi: \fs{E} \to \rr $ be a Keller's $ C^2_c $-mapping, and $ \mathcal{N} $ the Nehari
	manifold of $ \upphi $ that satisfies \eqref{eq:1000} and \eqref{eq:10001}. If $ \upphi $ satisfies condition 
	\eqref{s.1} at some point of $ \mathcal{N} $, then $ \upphi $ and $ \upphi|_{\mathcal{N}} $ have the same critical points.
\end{theorem}

\end{document}